\documentclass[a4paper,11pt,nosumlimits,nonamelimits]{article}
\usepackage[utf8]{inputenc}
\usepackage{amsmath,amssymb,amsfonts,amsthm}
\usepackage{array} 
\usepackage{url,hyperref,verbatim}
\usepackage{color}

\renewenvironment{itemize} 
	{\begin{list}
		{$\bullet$}{\setlength{\parskip}{5pt} \setlength{\topsep}{0cm}
		 \setlength{\partopsep}{0cm} \setlength{\itemsep}{3pt} \setlength{\parsep}{0cm}\item[]}}
	{\end{list}}

\theoremstyle{plain}
\newtheorem{theorem}{Theorem}

\newtheorem{assumption}{Assumption}
\newtheorem{proposition}{Proposition}

\theoremstyle{definition}
\newtheorem{definition}{Definition}

\theoremstyle{remark}
\newtheorem{remark}{Remark}

\newcommand{\real}{\mathbb{R}}
\newcommand{\nneg}{\mathbb{R}_+}

\newcommand{\ntr}{\mathbb{N}}

\newcommand{\vc}{\mathbf}
\newcommand{\wtil}{\widetilde}

\newcommand{\pr}{\mathbb{P}}
\newcommand{\ex}{\mathbb{E}}

\newcommand{\eps}{\varepsilon}

\newcommand{\lmb}{\lambda}

\newcommand{\dlt}{\delta}

\begin{document}

\abovedisplayskip=.6\abovedisplayskip 
\belowdisplayskip=.6\belowdisplayskip 
\abovedisplayshortskip=.6\abovedisplayshortskip 
\belowdisplayshortskip=.6\belowdisplayshortskip 

\frenchspacing 

\title{Fluid Limit of a PS-queue with Multistage Service}

\author{ Maria Remerova and Bert Zwart\footnote{MR~is with
Korteweg-de vries Institute for Mathematics, University of Amsterdam, P.O. Box 94248, 1090 GE Amsterdam, The Netherlands. E-mail:
\url{M.Remerova@uva.nl}. . BZ~is with CWI. E-mail: \url{Bert.Zwart@cwi.nl}.
BZ~is also affiliated with EURANDOM, VU University Amsterdam, and
Georgia Institute of Technology.}}

\date{\today}

\maketitle

\begin{abstract}
The PS-model treated in this paper is motivated by freelance job websites where multiple freelancers compete for a single job. In the context of such websites, multistage service of a job means collection of applications from multiple freelancers. Under Markovian stochastic assumptions, we develop fluid limit approximations for the PS-model in overload. Based on this approximation, we estimate what proportion of freelancers get the jobs they apply for. In addition, the PS-model studied here is an instant of PS with routing and impatience, for which no Lyapunov function is known, and we suggest some partial solutions.

\vspace{2ex} \textit{Keywords:} processor-sharing, routing, fluid limits, Lyapunov functions, freelance job websites.

\vspace{2ex} \textit{MSC2010:} Primary 60K25, 60F17; Secondary 90B15, 90B22.

\end{abstract}

\section{Introduction}
The original motivation for this paper lies in a model different than the one claimed in the title. Namely, we are interested in freelance job websites, which have two kinds of visitors: customers offering jobs and freelancers, or servers, looking for jobs. The key feature of such websites is that multiple servers compete for a single job there. The most common situation is competition at the stage of application, i.e.\ to get the job. Along with that, the applicants might have to do the job, and then the one who has done it best gets paid --- this is, for example, the former principle of work of \url{flightfox.com}, a website for searching cheap flight connections.

To start with, we design a~basic model of a freelance job website, where there is a Poisson stream of customers and a Poisson stream of freelancers of rates $\lambda$ and $\mu$, respectively. Each customer upon arrival posts a job on the website main page and sets a patience clock that is distributed exponentially with parameter $\nu$. Each freelancer upon arrival picks a job from the main page at random and applies for it, in the form of leaving a comment. If there are no jobs, the freelancer leaves. At most $I$ applications are allowed per job. Once a customer receives the $I$-th application, or his patience expires, he should remove the job from the main page and continue communication with the applicants via private messaging. The state of this system is represented by the vector composed of the numbers of jobs on the main page with $i = 0,\ldots, I-1$ applications. Mathematically, the constraint on the number of applications per job is convenient since it makes the problem finite-dimensional. In practice, such a threshold  is always present implicitly: jobs with too many applications are not attractive anymore since the chance to get them is small. We are not aware of websites with an explicit threshold, but we suggest it as a guarantee of a chance to get a job, which is necessary in case of the ``do-it-best-get-paid'' policy or in case the website aims to expand and thus encourage unexperienced freelancers to join.

Our next observation (inspired by Borst et al. \cite{BBMNQ}) is that the basic freelance model is actually equivalent to a PS-queue where
\begin{itemize}
\item arrivals are Poisson of rate~$\lambda$,
\item customers re-enter the queue for~$I$ times with independent service requirements distributed exponentially with parameter $\mu$,
\item patience times of customers are exponentially distributed with parameter $\nu$.
\end{itemize}
The state of the PS-queue is, respectively, the vector composed of the numbers of customers who have entered the queue for $i = 1, \ldots ,I$ times, or we also say "{\it customers at stage~$i$ of service}". Finally, we generalise the model by allowing service requirements at different stages of service have different parameters $\mu_i$ (the distribution is still exponential).

The results of this paper concern fluid limit approximations of the suggested model in overload. We show that trajectories of the per-stage population process, when scaled properly, converge to solutions of a system of differential equations, which in turn stabilise to the unique invariant solution over time. Then we use the fluid limit approximation to estimate the chance for a freelancer to get a job.

As for the proof techniques, convergence of the scaled trajectories follows by the classical arguments of compact containment and oscillation control. To establish convergence of fluid limits to the invariant point, we use an equivalent description of fluid limits, which is a generalisation of the approximating equation suggested by Gromoll et al. \cite{GRZ08} for a single-stage-service PS queue. We also discuss the method of of Lyapunov functions since the model treated in this paper is an instant of a more general open problem: no Lyapunov function is known for a PS-queue with routing and impatience. We present some partial solutions to this open problem.

In the future, we aim to build on the motivation behind this chapter. A next logical step would be to  incorporate the service stage in addition to the application stage. We are mostly interested in the scenario when the same job is done multiple times, which mathematically is a special kind of dependence of job sizes. There are also optimization questions that arise in practice. For example, if freelancers are ranked in a way, what strategies should they follow to build and maintain a strong reputation? The majority of freelance websites exist at the cost of transaction fees, then what are the ways to increase website profits while keeping transaction fees affordable to visitors? Recently there has been more interest in analysing related problems, cf. \cite{2},\cite{5}, \cite{3}, \cite{7},\cite{1}.

The paper is organised as follows. In Section~\ref{FL:sec:model_description}, we discuss in detail how the PS-queue with multistage service arises from the basic freelance model. In Section~\ref{FL:sec:fluid_model}, we introduce two equivalent deterministic systems of equations that are analogues of the stochastic model and discuss their properties. We also discuss Lyapunov functions and estimate the probability of a freelancer getting a job. Section~\ref{FL:sec:fluid_limit} specifies the fluid scaling under which the stochastic model converges to its deterministic analogues. In Sections~\ref{FL:sec:fluid_model_proof} and~\ref{FL:sec:fluid_limit_proof}, the proofs for the results of Sections~\ref{FL:sec:fluid_model} and~\ref{FL:sec:fluid_limit} are presented. Section~\ref{FL:appendix} shows how the convergence to the invariant point in the single-stage-service case implies that for the multistage-service case. In the remainder of this section we list the notation we use throughout the paper.

\paragraph{Notation} To define $x$ as equal to $y$, we write $x:=y$ or $y=:x$. We abbreviate the left-hand side and right-hand side of an equation as ``LHS" and ``RHS", respectively.

The standard sets are: the natural numbers $\ntr:= \{1,2, \ldots  \}$, the real line $\real := (-\infty, \infty)$ and non-negative half-line $\nneg:=[0,\infty)$.

All vector notations are boldface. Unless stated otherwise, the coordinates of an $I$ dimensional vector are denoted by the same symbol (regular font instead of bold) with subscripts $1, \ldots, I$ added. Overlining and superscripts of vectors remain in their coordinates as well, for example $\overline{\vc{Q}}^r(t) = (\overline{Q}_1, \ldots, \overline{Q}_I)(t)$. The $\real^I$ space is endowed with the $L_1$-norm $\|\vc{x} \|_1 := \sum_{i=1}^I |x_i|$.

For the metric space $S=\real$, $\nneg$ or $\nneg^I$, the notation  $\mathbf{D}(\nneg, S)$ stands for the space of functions $f \colon \nneg \to S$ that are right-continuous with left limits. This space is endowed with the Skorokhod $J_1$-topology.

\section{Stochastic model} \label{FL:sec:model_description} In this section we introduce two stochastic models: a basic model of a freelance job website  and a processor sharing (PS) queue with multistage service. We then discuss in what sense the latter model generalises the former.

\paragraph{Basic model of a freelance job website} There are two types of visitors on a freelance job website: customers, who publish job descriptions, and freelancers, who apply for those jobs. We assume that new jobs appear on the website main page according to a Poisson process of rate~$\lambda$, and that freelancers intending to find a job visit the website according to a Poisson process of rate~$\mu$. As a freelancer looking for a job visits the website, he picks a job from the main page at random and applies for it, say leaves a comment. Each job is allowed to collect at most~$I$ applications while its patience time lasts, measured from the moment the job was published and distributed exponentially with parameter~$\nu$. All the random elements mentioned: the arrival processes of jobs and freelancers, and patience times of different jobs are mutually independent. As soon as a~job either gets~$I$ applications, or its patience time expires, the customer removes the job description from the main page and continues communication with the applicants elsewhere. In this model, our focus is on the process 
\[
\vc{Q}^{\textup{FL}}(t) = (Q_0^{\textup{FL}}, \ldots, Q_{I-1}^{\textup{FL}})(t), \quad t \geq 0,
\]
where $Q_i^{\textup{FL}}(t)$ is the number of jobs on the main page that have collected $i$ applications up to time instant~$t$.

\paragraph{PS-queue with multistage service} Now consider a PS queue with Poisson arrivals of rate~$\lambda$. We assume that each customer of this queue should undergo~$I$ stages of service, with stage~$i+1$ starting immediately upon completion of stage~$i$ and the service requirement at stage~$i$ distributed exponentially with parameter~$\mu_i$. A customer is supposed to leave the queue upon service completion, but if his patience time expires earlier, he abandons then. As in the previous model, patience times are distributed exponentially with parameter~$\nu$. The arrival process, service requirements of all customers at all stages, and patience times of all customers are mutually independent. Here we analyse the process
\[
\vc{Q}(t) = (Q_1, \ldots, Q_I)(t), \quad t \geq 0,
\]
where $Q_i(t)$ stands for the number of customers in stage~$i$ of service at time instant~$t$.

\paragraph{Equivalence of the two models in case all $\mu_i$'s are the same} Suppose that, in the second model, all service stages have the same distribution parameter~$\mu$. In this case, the processes $\vc{Q}^{\textup{FL}}(\cdot)$ and $\vc{Q}(\cdot)$ are distributed identically. The idea is that the jobs waiting for the~$i$-th application (i.e.\ those with~$i-1$ applications) can be viewed as customers of the PS-queue who are undergoing stage~$i$ of service, and the moments jobs receive applications --- as completions of stages of service in the PS-queue. When a freelancer applies for a job, he picks one at random. Correspondingly, if there is a service stage completion in the PS-queue, all of the service stages that have been ongoing are equally likely to be the one that has finished. That is due to the memoryless property of the exponential distribution and because all the $\mu_i$'s are the same.

The above insight originally belongs to Borst et al. \cite{BBMNQ}, who discussed the equivalence of PS and random order of service in the context of the $G/M/1$ queue. To formalise the idea they constructed a probabilistic coupling, which can be generalised in a straightforward way to the two models we consider here.

\paragraph{Dynamic equations} Most of the results presented in this paper are developed for the more general model of PS with multistage service. We assume it is defined on a probability space $(\Omega, \mathcal{F}, \pr)$ with expectation operator $\ex$. Denote the arrival process of customers, which is Poisson of rate~$\lambda$, by $A(\cdot)$. These are arrivals to stage~$1$ of service. Let $D_i^\textup{s}(\cdot)$ stand for the process of service completions at stage~$i$. Note that, for $i \leq I-1$, $D_i^\textup{s}(\cdot)$ is the arrival process to stage~$i+1$, and $D_I^\textup{s}(\cdot)$ is the process of departures due to total service completions. Finally, denote by $D_i^\textup{a}(\cdot)$ the process of abandonments due to impatience at stage~$i$. Since service requirements at all stages and patience times of all customers are distributed exponentially, and since the exponential distribution is memoryless, the processes  $D_i^\textup{s}(\cdot)$ and $D_i^\textup{a}(\cdot)$ are doubly stochastic Poisson with instantaneous rates $\mu_i Q_i (\cdot) / \| Q (\cdot) \|$ (zero by convention when the system is empty) and $\nu Q_i(\cdot)$, respectively. That is, the population process $\vc{Q}(\cdot) = (Q_1, \ldots, Q_I)(\cdot)$ can be represented as the unique (see e.g.~\cite{MMR98}) solution to the following system of equations: for $t \geq 0$,
\begin{equation} \label{FL:eq:dyn_1}
\begin{split}
Q_1(t) &= Q_1(0) + A(t) - D_1^\textup{s}(t) - D_1^\textup{a}(t), \\
Q_i(t) &= Q_i(0) + D_{i-1}^\textup{s}(t) - D_i^\textup{s}(t) - D_i^\textup{a}(t), \quad i \geq 2,
\end{split}
\end{equation}
with
\begin{equation} \label{FL:eq:dyn_2}
\begin{split}
D_i^\textup{s}(t) &= \Pi_i^\textup{s} \Bigl( \mu_i \int_0^t \frac{Q_i(u)}{\|\vc{Q}(u)\|_1} \, du \Bigr),\\
 D_i^\textup{a}(t) &= \Pi_i^\textup{a} \Bigl( \nu \int_0^t Q_i(u) du \Bigr),
\end{split}
\end{equation}
where $\Pi_i^\textup{s}(\cdot), \Pi_i^\textup{a}(\cdot)$ are Poisson processes of unit rate for all~$i$, and also the initial state $\vc{Q}(0) = (Q_1, \ldots, Q_I)(0)$, the arrival process $A(\cdot)$ and the processes $\Pi_i^\textup{s}(\cdot)$, $\Pi_i^\textup{a}(\cdot)$ are mutually independent.

Finally, throughout the rest of the paper, we assume the following.

\begin{assumption} \label{FL:ass:load} The system is overloaded, i.e.\ $\lambda \sum_{i=1}^I 1/\mu_i > 1$.
\end{assumption}

\section{Fluid model} \label{FL:sec:fluid_model}
In this section, we define and analyse a fluid model --- a deterministic analogue of the PS-queue with multistage service introduced above. We use the fluid model to estimate the chance of a freelancer getting a job when the application limit $I$ is large. In the next section, the fluid model is shown to approximate the stochastic PS model, where the time and space are appropriately normalised.

\begin{definition}
A function $\vc{z}(\cdot) = (z_1, \ldots, z_I)(\cdot) \colon \nneg \to \real_+^I$ that is continuous and such that  $\inf_{t \geq \delta} \| \vc{z}(t) \|_1 > 0$ for any $\delta > 0$ is called a {\it fluid model solution (FMS)} if it solves the following system of differential equations: for $t>0$,
\begin{equation} \label{FL:eq:fluid_model_dif}
\begin{split}
z_1'(t)  &= \lambda - \mu_1 \dfrac{z_1(t)}{\|\vc{z}(t)\|_1} - \nu z_1(t), \\
z_i'(t)  &= \mu_{i-1} \dfrac{z_{i-1}(t)}{\|\vc{z}(t)\|_1} - \mu_i \dfrac{z_i(t)}{\|\vc{z}(t)\|_1} - \nu z_i(t), \quad i \geq 2.
\end{split}
\end{equation}
\end{definition}

When investigating properties of FMS's, we will also use an alternative description of them. Let random variables $B_i$, $i= 1, \ldots, I$, and $D$, all defined on the same probability space $(\Omega, \mathcal{F}, \pr)$, be mutually independent and distributed exponentially, $B_i$ with parameter~$\mu_i$ for all~$i$ and $D$ with parameter~$\nu$. Introduce also
\[
B_j^i := \left\{
\begin{array}{ll}
\sum_{l=j}^i B_l, & j \leq i, \\
0, & j > i.
\end{array}
\right.
\]
It turns out that~\eqref{FL:eq:fluid_model_dif} is equivalent to the following system of integral equations: for $i = 1, \ldots, I$ and $t \geq 0$,
\begin{equation} \label{FL:eq:fluid_model_int}
\begin{split}
z_i(t) =& \sum_{j=1}^i z_j(0) \, \pr \Bigl\{ B_j^{i-1} \leq \int_0^t \frac{du}{\|\vc{z}(u)\|_1} < B_j^i, \, D > t \Bigr\} \\
&+ \lambda \int_0^t \pr \Bigl\{ B_1^{i-1} \leq \int_s^t \frac{du}{\|\vc{z}(u)\|_1} < B_1^i, \, D > t-s \Bigr\} \, ds.
\end{split}
\end{equation}
The two systems are equivalent in the sense that they have the same set of continuous, non-negative, non-zero outside $t=0$ solutions.

The differential equations \eqref{FL:eq:fluid_model_dif} capture the drift of the system, they are direct analogues of the stochastic equations \eqref{FL:eq:dyn_1}--\eqref{FL:eq:dyn_2}. The integral equations~\eqref{FL:eq:fluid_model_int} mimic the evolution of the stochastic system from an individual customer's prospective. Given a customer arrived at time instant~$s$, he is undergoing stage~$i$ of service at time instant $t \geq s$ if his patience time allows it, and if the amount of service he has received up to~$t$ covers the service requirements of the first $i-1$ stages completely and the service requirement of stage~$i$ only partially. This explains the second term in the RHS of~\eqref{FL:eq:fluid_model_int}. The first term has the same interpretation but in the context of customers who were present in the system at $t= 0$. Due to the memoryless property of the exponential distribution, the residual service requirements of the service stages that are ongoing at $t = 0$ are still exponentially distributed with the corresponding parameters.

A rigorous proof of the equivalence of the two descriptions of FMS's follows in Section~\ref{FL:sec:fluid_model_proof}. It exploits certain properties of the exponential and phase-type distributions.

We now proceed with the analysis of FMS's.

\begin{theorem} For any initial state $\vc{z}(0)$, a FMS exists and is unique.
\end{theorem}

\begin{proof}
Existence of FMS's is established in Sections~\ref{FL:sec:fluid_limit} and~\ref{FL:sec:fluid_limit_proof}: fluid limits of the population process~$\vc{Q}(\cdot)$ are FMS's. When proving uniqueness, we distinguish between two cases. If the initial state is non-zero, the uniqueness follows from the description~\eqref{FL:eq:fluid_model_dif} by the Gronwall inequality: it applies because the RHS of \eqref{FL:eq:fluid_model_dif} is Lipschitz continuous on sets $\{ \vc{z} \in \nneg^I \colon ||\vc{z}||_1 \geq a \}$, $a>0$. In case the initial state is zero, we use the description~\eqref{FL:eq:fluid_model_int}. The summation of the equations of~\eqref{FL:eq:fluid_model_int} where $\vc{z}(0) = \vc{0}$ implies that the norm $\|\vc{z}(\cdot)\|_1$ solves the following equation: for $t \geq 0$,
\begin{equation} \label{FL:eq:norm_fms_int}
x(t) = \lambda \int_0^t \pr \Bigl\{ B_1^I > \int_s^t \frac{du}{x(u)}, \, D > t-s \Bigr\} \, ds.
\end{equation}
The last equation is, in fact, the fluid model of a PS-queue with single-stage service. It is studied in~Gromoll et al. \cite{GRZ08} and shown to have a unique solution that is bounded away from zero outside $t = 0$, see Corollary~3.8. So the norm $\|\vc{z}(\cdot)\|_1$ is unique. Then a solution to~\eqref{FL:eq:fluid_model_int} must be unique as well, since the individual coordinates $z_i(\cdot)$ are uniquely defined by the norm~$\| \vc{z}(\cdot) \|_1$ in~\eqref{FL:eq:fluid_model_int}. 
\end{proof}

In the next theorem we characterize the invariant (constant) FMS.

\begin{theorem}
There exists a unique invariant FMS, which is given by
\begin{equation} \label{FL:eq:0}
\begin{split}
z_1^\ast &= \dfrac{\lambda}{\mu_1 + \nu \|\vc{z}^\ast\|_1} \,\|\vc{z}^\ast\|_1, \\
z_i^\ast &= \dfrac{\mu_{i-1}}{\mu_i + \nu \|\vc{z}^\ast\|_1}\,  z_{i-1}^\ast, \quad i \geq 2,
\end{split}
\end{equation}
where $\|\vc{z}^\ast\|_1$ solves
\begin{equation} \label{FL:eq:2}
\begin{split}
f(\|\vc{z}^\ast\|_1) := \lambda \biggl( \frac{1}{\mu_1 + \nu \|\vc{z}^\ast\|_1} + \frac{\mu_1}{(\mu_1 + \nu \|\vc{z}^\ast\|_1) (\mu_2 + \nu \|\vc{z}^\ast\|_1)}&  \\
+ \cdots+ \frac{\mu_1 \ldots \mu_{I-1}}{(\mu_1 + \nu \|\vc{z}^\ast\|_1) \ldots (\mu_I + \nu \|\vc{z}^\ast\|_1)} &\biggr) = 1.
\end{split}
\end{equation}
\end{theorem}

\begin{proof}
By definition, an invariant FMS must be non-zero. It follows from the description~\eqref{FL:eq:fluid_model_dif} that an invariant FMS $\vc{z}^\ast = (z_1^\ast, \ldots, z_I^\ast)$ is defined by the following system of equations:
\begin{equation} \label{FL:eq:fp_dif}
\begin{split}
&\lambda - \mu_1 \frac{z_1^\ast}{\|\vc{z}^\ast\|_1} - \nu z_1^\ast = 0, \\
&\mu_{i-1} \frac{z_{i-1}^\ast}{\|\vc{z}^\ast\|_1} - \mu_i \frac{z_i^\ast}{\|\vc{z}^\ast\|_1} - \nu z_i^\ast = 0, \quad i \geq 2.
\end{split}
\end{equation}

As we solve the $i$-th equation in~\eqref{FL:eq:fp_dif} with respect to~$z_i^\ast$, we obtain~\eqref{FL:eq:0}.

Now,~\eqref{FL:eq:0} is equivalent to
\begin{equation*}
\begin{split}
z_1^\ast &= \dfrac{\lambda}{\mu_1 + \nu \|\vc{z}^\ast\|_1} \, \|\vc{z}^\ast\|_1, \\
z_i^\ast &= \dfrac{\mu_{i-1} \ldots \mu_1}{(\mu_i + \nu \|\vc{z}^\ast\|_1) \ldots (\mu_2 + \nu \|\vc{z}^\ast\|_1)} \, z_1^\ast, \quad i \geq 2.
\end{split}
\end{equation*}

As we sum up over the last set of equations and divide by $\|\vc{z}^\ast\|_1$ on both sides,~\eqref{FL:eq:2} follows.

Note that equations~\eqref{FL:eq:0}--\eqref{FL:eq:2} have a unique solution. Indeed, the  function $f(\cdot)$ is strictly decreasing in $(0,\infty)$ and takes all values between $\lambda \sum_{i=1}^I 1/\mu_i$ (which is bigger than~$1$ by Assumption~\ref{FL:ass:load}) and 0 as it arguments runs from $0$ to $\infty$. Hence \eqref{FL:eq:2} uniquely defines the norm $\|\vc{z}^\ast\|_1$, and then \eqref{FL:eq:2} uniquely defines the individual coordinates $z_i$ via $\|\vc{z}^\ast\|_1$.
\end{proof}

Finally, we show that the invariant FMS found above is asymptotically stable.

\begin{theorem} \label{FL:th:fixed_point_stable}
Any FMS $\vc{z}(t)$ converges to the unique invariant FMS $\vc{z}^\ast$ as $t \to \infty$.
\end{theorem}

When proving the last theorem, we again refer to the paper~\cite{GRZ08} on PS with single-stage service. The equations of~\eqref{FL:eq:fluid_model_int} summed up give: for $t \in \nneg$,
\begin{equation} \label{FL:eq:fms_norm}
\begin{split}
\|\vc{z}(t)\|_1 =& \sum_{j=1}^I z_j(0) \, \pr \Bigl\{B_j^I > \int_0^t \frac{du}{\|\vc{z}(u)\|_1}, D>t \Bigr\} \\
&+ \lambda \int_0^t \pr \Bigl\{ B_1^I > \int_s^t \frac{du}{\|\vc{z}(u)\|_1}, \, D > t-s \Bigr\} \, ds.
\end{split}
\end{equation}
In the last equation, we put $\vc{z}(\cdot) \equiv \vc{z}^\ast$ and take $t \to \infty$, which implies that the norm $\|\vc{z}^\ast\|_1$ of the invariant FMS should solve the equation
\begin{equation} \label{FL:eq:fp_norm}
x = \lambda \ex \min\{ x(B_1 + \ldots + B_I), D \}.
\end{equation}
Gomoll et al.\ \cite{GRZ08} show that \eqref{FL:eq:fp_norm} has a unique solution, so it must be $\|\vc{z}^\ast\|_1$; see Theorem~2.4 in~\cite{GRZ08}. It also follows from Theorem~2.4 that all solutions $\|\vc{z}(t)\|_1$ to~\eqref{FL:eq:fms_norm} converge to the unique solution $\|\vc{z}^\ast\|_1$ of~\eqref{FL:eq:fp_norm} as $t \to \infty$. To be precise, the theorem works with a slightly different equation than~\eqref{FL:eq:fms_norm}, but the difference is in the terms that represent the initial customers and vanish as $t \to \infty$. Now that we have the convergence of the norm $\|\vc{z}(t)\|_1 \to \|\vc{z}^\ast\|_1$ for any FMS $\vc{z}(\cdot)$, the coordinate-wise convergence can be shown with the use of the same ideas as in Theorem~2.4 of Gromoll et al.\ \cite{GRZ08}. We provide the proof in Section~\ref{FL:appendix} for completeness.

{\begin{remark}[Asymptotic stability of the invariant point via Lyapunov functions]
An alternative way to establish the asymptotic stability of the invariant solution to the fluid model~\eqref{FL:eq:fluid_model_dif} would be to suggest a Lyapunov function, i.e.\ a function $L \colon (0,\infty)^I \to \nneg$ such that $L(\vc{z}) \to \infty$ as $\|\vc{z}\|_1 \to \infty$ and whose derivative with respect to~\eqref{FL:eq:fluid_model_dif} is non-positive. It is known that a PS-queue with $I$ classes of customers, a Markovian routing and no impatience admits the entropy Lyapunov function (see Bramson \cite{Bram96})
\begin{equation} \label{eq:Lyapunov_entropy}
L_\textup{lg}(\vc{z}) := \sum_{i=1}^I z_i \ln \left( \frac{z_i / \|\vc{z}\|_1}{z_i^\ast/\|\vc{z}^\ast\|_1} \right).
\end{equation}
It can also be checked along the lines of Theorem~\cite[Theorem ]{ALOHA}  that a PS-queue with impatience (different rates $\nu_i$ for different classes are allowed) and no routing admits the quadratic Lyapunov function
\[
{L}_\textup{qd}(\vc{z}) = \sum_{i=1}^I \dfrac{(z_i-z_i^\ast)^2}{\mu_i z_i^\ast / \|\vc{z}^\ast\|_1}.
\]
Whether there is a Lyapunov function for a PS-queue with both routing and impatience is an open problem. In the particular case of a PS-queue with multistage service, where the routing is tandem (from class~$i$ to $i+1$) and the impatience parameters are the same for all classes, the mentioned open problem does not seem to become easier. We have, however, come up with some partial solutions, which we present here without a proof. For a general Markovian routing and the same impatience parameters for all classes, the entropy Lyapunov function \eqref{eq:Lyapunov_entropy} works if there are  $I=2$ classes of customers, and if there are $I >2$ classes, it can be shown to work  everywhere except for a~compact set (the derivative of $L_\textup{lg}(\vc{z})$ with respect to~\eqref{FL:eq:fluid_model_dif} is non-negative on $\{\vc{z} \in (0,\infty)^I \colon \|\vc{z}\|_1 \geq \|\vc{z}^\ast\|_1$). In case of $I=2$ classes, a general Markovian routing $(P_{i,j})_{i,j=1}^2$ and different impatience parameters $\nu_1, \nu_2$, the following quadratic Lyapunov function works:
\[
L_\textup{qd}(\vc{z}) = \alpha_1 (z_1 - z_1^\ast)^2 + \alpha_2 (z_2 - z_2^\ast)^2,
\]
where
\[
\alpha_1 = \frac{1}{[(1 - P_{1,1}) \mu_1 + P_{2,1} \mu_2 ] q_1}, \quad
\alpha_2 = \frac{1}{[(1 - P_{2,2}) \mu_2  + P_{1,2} \mu_1 ] q_2}.
\]
\end{remark}}

\paragraph{Probability for a freelancer to get a job} In the next section we discuss in what sense the fluid model approximates the stochastic PS-model. Here we estimate the chance of a freelancer getting a job based on the fluid model and under the following additional assumptions:
\begin{itemize}
\item[(A1)] the application limit $I$ is large,
\item[(A2)] all freelancers that applied for the same job have equal chances to get the job.
\end{itemize}
For the basic model of a freelance website presented in the previous section, the invariant point equations \eqref{FL:eq:0}--\eqref{FL:eq:2} 
can be rewritten as
\[
z_i^\ast = \frac{\lmb}{\nu} u^i (1-u), \quad 0 \leq i \leq I-1,
\]
where $z_i^\ast$ stands for the number of jobs on the website with $i$ applications and $u$ is the unique solution to
\begin{equation}
\frac{\lmb}{\mu} \sum_{i=1}^I u^i = 1. \label{eq:0}
\end{equation}
(We make the substitution $u:=\mu/(\mu+\nu \| \vc{z}^\ast\|_1)$ to obtain the equation for $u$.)

Now, by \eqref{FL:eq:fp_dif}, out of $\lmb$ jobs arriving per time unit, the fraction $\nu z_i^\ast / \lmb$ will leave with $i=0, \ldots, I-1$ applications due to impatience, and the fraction $\mu z_{I-1}^\ast / (\lmb \|\vc{z}^\ast\|_1)$ will be patient enough to collect $I$ applications. Then, by the assumption (A2), the probability for a freelancer to get a job under the application limit $I$ is
\[
P_I = \sum_{i=1}^{I-1} \frac{1}{i} \frac{\nu z_i^\ast}{\lmb} + \frac{1}{I} \frac{\mu z_{I-1}^\ast}{\lmb \|\vc{z}^\ast \|_1} = \sum_{i=1}^I \frac{1}{i} u^i (1-u) + \frac{1}{I} u^I.
\]
As $I \to \infty$, the solution $u$ to \eqref{eq:0} and $P_I$ converge:
\begin{align*}
u &\to u_\infty := \frac{\mu}{\lmb+\mu}, \\
P_I & \to \sum_{i=1}^\infty \frac{1}{i} u_\infty^i (1-u_\infty) = --  (1-u_\infty)\ln (1-u_\infty).
\end{align*}
Hence, for a large application limit $I$, the probability to get a job can be approximated by
\[
P_\infty =  - \frac{\lmb}{\lmb+\mu} \ln \frac{\lmb}{\lmb+\mu}.
 \]
 Interestingly, neither $P_I$ nor $P_\infty$ depend on the impatience parameter $\nu$. Intuitively, bigger $\nu$'s, i.e. smaller patience times, should result in jobs collecting less applications meaning less competition among freelancers and a higher chance to get a job.

\section{Fluid limit theorem} \label{FL:sec:fluid_limit}
In this section, we show that, under a proper scaling, the PS-queue with multistage service converges to the fluid model introduced in Section~\ref{FL:sec:fluid_model}.

Consider a family of stochastic PS-queues upper-indexed by positive numbers $r$, all of them defined on the probability space $(\Omega, \mathcal{F}, \pr)$. Let the arrival rate $\lambda$ and the parameters $\mu_i$ of service stages be the same in all models (and satisfy Assumption~\ref{FL:ass:load}), and the impatience parameter of model~$r$ be $\nu/r$. Define the fluid scaled population processes
\begin{equation} \label{FL:eq:scaled_process}
\overline{\vc{Q}}^{\,r}(t) := \vc{Q}^r(rt) / r, \quad t \in \nneg.
\end{equation}

We refer to weak limits along subsequences of the processes~\eqref{FL:eq:scaled_process} as {\it fluid limits}. They can be characterised as solutions to the differential/integral equations \eqref{FL:eq:fluid_model_dif}--\eqref{FL:eq:fluid_model_int}, as the next theorem asserts.

\begin{theorem} \label{FL:th:fluid_limit} Suppose that $\overline{\vc{Q}}^{\,r}(0) \Rightarrow \vc{z}(0)$ as $r \to \infty$, where $\vc{z}(0)$ is a random vector. Then the processes $\overline{\vc{Q}}^{\,r}(\cdot)$ converge weakly in the Skorokhod space $\mathbf{D}(\nneg, \nneg^I)$ to the unique FMS with initial state $\vc{z}(0)$.
\end{theorem}

The proof is given in Section~\ref{FL:sec:fluid_limit_proof}. First, we show that the family of the scaled processes~\eqref{FL:eq:scaled_process} is relatively compact by checking the compact containment and oscillation control conditions. That is, we show that fluid limits exist. Then we check that they are FMS's by deriving the fluid model equations~\eqref{FL:eq:fluid_model_dif} from the scaled stochastic dynamic equations~\eqref{FL:eq:dyn_1}--\eqref{FL:eq:dyn_2}.

\section{Equivalence of the two fluid model descriptions} \label{FL:sec:fluid_model_proof}
This proof partly relies on the ideas of the proof of a similar result in \cite{ALOHA}, see Lemma~2, but it is more involved. In particular, it uses (and establishes) the special property~\eqref{FL:eq:5} of the phase-type distribution.

Let a function $\vc{z} \colon \nneg \to \real_+^I$ be continuous and non-zero outside $t = 0$. 

\paragraph{Proof of \eqref{FL:eq:fluid_model_dif} $\Rightarrow$ \eqref{FL:eq:fluid_model_int}}  Suppose that $\vc{z}(\cdot)$ is a solution to~\eqref{FL:eq:fluid_model_dif}. Consider the following Cauchy problem with respect to $\vc{u}(\cdot)$: for $t>0$,
\begin{equation} \label{FL:eq:3}
\begin{split}
u'_1(t) &= \lmb - \mu_1 \dfrac{u_1(t)}{\|\vc{z}(t)\|_1} - \nu u_1(t),\\
u'_i(t) &= \mu_{i-1} \dfrac{u_{i-1}(t)}{\|\vc{z}(t)\|_1} - \mu_i \dfrac{u_i(t)}{\|\vc{z}(t)\|_1} - \nu u_i(t), \quad i \geq 2, \\
\vc{u}(0) &= \vc{z}(0).
\end{split}
\end{equation}
This problem has at most one continuous solution. Indeed, let $\vc{u}(\cdot)$ and $\wtil{\vc{u}}(\cdot)$ be two continuous solutions to~\eqref{FL:eq:3}. Then the difference $\vc{w}(\cdot) := (\vc{u} - \wtil{\vc{u}})(\cdot)$ satisfies: for $t>0$,
\begin{equation*}
\begin{split}
w'_1(t) &= - w_1(t) \left(\frac{\mu_1}{\|\vc{z}(t)\|_1} + \nu \right),\\
w'_i(t) &=  w_{i-1}(t)\frac{\mu_{i-1}}{\|\vc{z}(t)\|_1} - w_i(t) \left(\frac{\mu_i}{\|\vc{z}(t)\|_1} + \nu \right), \quad i \geq 2, \\
\vc{w}(0) &= \vc{0}.
\end{split}
\end{equation*}
Note that if $w_1(t) > 0$, then $w'_1(t) < 0$, and the other way around. Then $w_1(\cdot) \equiv 0$ (see e.g.\ \cite[Lemma 1]{ALOHA}) and $w'_2(t) = - w_2(t) (\mu_i / \|\vc{z}(t)\|_1 + \nu )$, $t>0$. To each pair $w_i(\cdot)$ and $w_{i+1}(\cdot)$, we apply the same reasoning as to $w_1(\cdot)$ and $w_2(\cdot)$, and thus obtain $\vc{w}(\cdot) \equiv \vc{0}$.

It is straightforward to check that the LHS and RHS of~\eqref{FL:eq:fluid_model_int} both satisfy~\eqref{FL:eq:3}. Since a solution to~\eqref{FL:eq:3} must be unique, the LHS and RHS of~\eqref{FL:eq:fluid_model_int} must coincide.

\paragraph{Proof of \eqref{FL:eq:fluid_model_int} $\Rightarrow$ \eqref{FL:eq:fluid_model_dif}} Suppose now that $\vc{z}(\cdot)$ solves~\eqref{FL:eq:fluid_model_int}. As we differentiate the RHS of~\eqref{FL:eq:fluid_model_int}, it follows that, for $t > 0$,
\begin{equation*}
\begin{split}
z_1'(t)  &= \lambda - \mu_1 \frac{z_1(t)}{\|\vc{z}(t)\|_1} - \nu z_1(t), \\
z_i'(t)  &= \displaystyle{\sum_{j=1}^i} \left(f_{B_j^{i-1}} - f_{B_j^i}\right) \left( \int_0^t \dfrac{du}{\|\vc{z}(u)\|_1} \right) \dfrac{\pr\{D>t\}}{\|\vc{z}(t)\|_1} \,  \\
&\phantom{=}+ \lmb \displaystyle{\int_0^t} \left(f_{B_1^{i-1}} - f_{B_1^i}\right) \left( \int_s^t \dfrac{du}{\|\vc{z}(u)\|_1} \right) \dfrac{\pr\{D>t-s\}}{\|\vc{z}(t)\|_1} \,  \, ds \\
&\phantom{=}- \nu z_i(t), \quad i \geq 2,
\end{split}
\end{equation*}
where $f_{B_j^i}(\cdot)$ denotes the probability density function of the phase-type random variable
\[
B_j^i := \left\{
\begin{array}{ll}
\sum_{l=j}^i B_l, & j \leq i, \\
0, & j > i,
\end{array}
\right.
\]
where $B_l$ is distributed exponentially with parameter $\mu_l$.

At this stage, in order to have~\eqref{FL:eq:fluid_model_dif}, it suffices to show that, for $t>0$,
\begin{equation} \label{FL:eq:4}
\begin{split}
\mu_i z_i(t) =& \sum_{j=1}^i f_{B_j^i} \left( \int_0^t \dfrac{du}{\|\vc{z}(u)\|_1} \right) \pr\{D>t\} \\
&+ \lmb \int_0^t f_{B_1^i} \left( \int_s^t \dfrac{du}{\|\vc{z}(u)\|_1} \right) \dfrac{\pr\{D>t-s\}}{\|\vc{z}(t)\|_1} \,  ds.
\end{split}
\end{equation}
In turn, in order to have~\eqref{FL:eq:4} under the assumption~\eqref{FL:eq:fluid_model_int}, it suffices to show that, for all $i \in \ntr$ and $x \in \real$,
\begin{equation*}
\dfrac{1}{\mu_i} f_{B_1^i}(x) = \pr \{ B_1^{i-1} \leq x < B_1^i \},
\end{equation*}
or equivalently,
\begin{equation} \label{FL:eq:5}
\pr \{ B_1^i > x \} = \sum_{j=1}^i \dfrac{1}{\mu_j} f_{B_1^j}(x).
\end{equation}
(To be precise, for $i>I$, we need to introduce random variables $B_i$ distributed exponentially with parameters~$\mu_i$, mutually independent with each other and with $B_j$, $j \leq I$.)

We prove~\eqref{FL:eq:5} by induction: it holds for $i=1$, assume that it holds for an $i \geq 1$, we have to check that it holds for $i+1$ as well. By the convolution formula,
\begin{align*}
\pr \{ B_1^{i+1} > x \} 
 &= \int_0^\infty \pr \{(y + B_2^{i+1} > x\} f_{B_1}(y) dy \\
&=  \int_x^\infty f_{B_1}(y) dy + \int_0^x \pr \{B_2^{i+1} > x-y\} f_{B_1}(y) dy.
\end{align*}
Now we incorporate the induction hypothesis and obtain
\begin{align*}
\pr \{ B_1^{i+1} > x \} 
&= \pr\{ B_1 > x\} + \sum_{j=2}^{i+1} \dfrac{1}{\mu_{j}} \int_0^x f_{B_2^j}(x - y)f_{B_1}(y) dy \\
&= \dfrac{1}{\mu_1} f_{B_1}(x) + \sum_{j=2}^{i+1} \dfrac{1}{\mu_{j}} \int_{-\infty}^\infty f_{B_2^j}(x - y)f_{B_1}(y) dy \\
&=  \dfrac{1}{\mu_1} f_{B_1}(x) + \sum_{j=2}^{i+1} \dfrac{1}{\mu_{j}} f_{B_1^j}(x).
\end{align*}
So~\eqref{FL:eq:5} indeed holds and implies~\eqref{FL:eq:4}; and \eqref{FL:eq:4}, in turn, implies~\eqref{FL:eq:fluid_model_dif}.

\section{Proof of Theorem~\ref{FL:th:fixed_point_stable}} \label{FL:appendix}

It follows from the fluid model description~\eqref{FL:eq:fluid_model_int}, that the coordinates of the invariant FMS $\vc{z}^\ast$ are uniquely defined by its norm via
\begin{equation} \label{FL:eq:fp_int}
z_i^\ast = \lambda \ex \min\{\| \vc{z}^\ast \|_1 B_1^i, D \}  - \lambda \ex \min\{\| \vc{z}^\ast \|_1 B_1^{i-1}, D \} \quad \text{for all $i$}.
\end{equation}

It is shown in~Gromoll et al. \cite[Theorem 2.4]{GRZ08} that, for any FMS $\vc{z}(\cdot)$, we have $\|\vc{z}(t)\|_1 \to \|\vc{z}^\ast\|_1$ as $t \to \infty$. Here we derive the coordinate-wise convergence from the convergence of the norms.

As we compare~\eqref{FL:eq:fluid_model_int} to \eqref{FL:eq:fp_int}, it follows that, in order to have $z_i(t) \to z_i^\ast$ as $t \to \infty$, it suffices to show that, for all $i$,
\begin{equation} \label{FL:eq:25}
\int_0^t f_i(s,t) \,ds \\
\to \ex \min\{\| \vc{z}^\ast \|_1 B_1^i, D \},
\end{equation}
where
\[
f_i(s,t) = \pr \Bigl\{ B_1^i > \int_s^t \frac{du}{\|\vc{z}(u)\|_1}, D>t-s \Bigr\}.
\]

Fix an $\eps \in (0,\|\vc{z}^\ast\|_1)$ and let $t_\eps$ be such that
\[
\|\vc{z}^\ast\|_1 - \eps \leq \|\vc{z}(t)\|_1 \leq \|\vc{z}^\ast\|_1 + \eps \quad \text{for all $t \geq t_\eps$}.
\]

For any fixed $s$, $f_i(s,t) \to 0$ as $t \to \infty$, and then, by the dominated convergence theorem,
\begin{equation} \label{FL:eq:24}
\int_0^{t_\eps} f_i(s,t) \, ds \to 0 \quad \text{as $t \to \infty$}.
\end{equation}

For all $t \geq t_\eps$, we have
\begin{align*}
\int_{t_\eps}^t f_i(s,t) \, ds &\leq \int_{t_\eps}^t \pr \Bigl\{ B_1^i > \int_s^t \frac{du}{\|\vc{z}^\ast\|_1+\eps}, \, D>t-s \Bigr\} \, ds \\
&\leq \int_{0}^{t - t_\eps} \pr \left\{ \min\{ (\|\vc{z}^\ast\|_1+\eps)B_1^i, \,D \} \geq s\right\} \, ds,
\end{align*}
which, in combination with~\eqref{FL:eq:24}, implies that
\[
\limsup_{t \to \infty} \int_0^t f_i(s,t) \, ds \leq \ex \min\{ (\|\vc{z}^\ast\|_1+\eps) B_1^i, \, D \}.
\]

Similarly, we obtain
\[
\liminf_{\,t \to \infty} \int_0^t f_i(s,t) \, ds \geq \ex \min\{ (\| \vc{z}^\ast\|_1-\eps) B_1^i, \, D \}.
\]

As we take $\eps \to 0$ in the last two equations, \eqref{FL:eq:25} follows.

\section{Proof of Theorem~\ref{FL:th:fluid_limit}} \label{FL:sec:fluid_limit_proof}
The proof consists of two parts. First we show that the family of the fluid scaled processes $\overline{\vc{Q}}^{\, r}(\cdot)$ is $\mathbf{C}$-tight, i.e.\ that fluid limits exist and are continuous. Then we check that fluid limits are FMS's, i.e.\ that they are bounded away from zero outside $t=0$ and solve the fluid model equations~\eqref{FL:eq:fluid_model_dif}.

Throughout the proof, we use the following representation of the processes~$\vc{Q}^r(\cdot)$ (they all are defined on the same probability space $(\Omega, \mathcal{F}, \pr)$): for $t \in \nneg$,
\begin{equation} \label{FL:eq:dyn_1_r}
\begin{split}
Q_1^r(t) &= Q_1^r(0) + A(t) - D_1^{r, \textup{s}}(t) - D_1^{r,\textup{a}}(t), \\
Q_i^r(t) &= Q_i^r(0) + D_{i-1}^{r,\textup{s}}(t) - D_i^{r,\textup{s}}(t) - D_i^{r,\textup{a}}(t), \quad i \geq 2,
\end{split}
\end{equation}
with
\begin{equation} \label{FL:eq:dyn_2_r}
\begin{split}
D_i^{r,\textup{s}}(t) &= \Pi_i^\textup{s} \left( \mu_i \int_0^t \frac{Q^r_i(u)}{\|\vc{Q}^r(u)\|_1}  \,du \right), \\
D_i^{r,\textup{a}}(t) &= \Pi_i^\textup{a} \left( \dfrac{\nu}{r} \int_0^t Q^r_i(u) du \right),
\end{split}
\end{equation}
where the processes $A(\cdot)$ and $\Pi_i^\textup{s}(\cdot), \Pi_i^\textup{a}(\cdot)$ are the same as in \eqref{FL:eq:dyn_1}--\eqref{FL:eq:dyn_2}, except that this time we assume them to be independent from the family of the initial states $\vc{Q}^r(0)$.

\paragraph{$\mathbf{C}$-tightness} In order to prove that the family of the processes $\overline{\vc{Q}}^{\,r}(\cdot)$ is $\mathbf{C}$-tight, it suffices to show that the following two properties hold (see Ethier and Kurtz~\cite{EthierKurtz}): for any $T > 0$ and $\eps > 0$, there exist an $M < \infty$ and a $\delta > 0$ such that
\begin{equation} \label{FL:eq:compact}
\liminf_{\,r \to \infty} \pr \{ \| \overline{\vc{Q}}^{\,r}(T) \|_1 \leq M \} \geq 1 - \eps,
\end{equation}
and
\begin{equation} \label{FL:eq:oscillation}
\liminf_{\,r \to \infty} \pr \{ \sup_{\begin{subarray}{l} s,t\in [0,T], \\ |s-t|<\delta \end{subarray}} \|\overline{\vc{Q}}^{\,r}(s) - \overline{\vc{Q}}^{\,r}(t) \|_1 \leq \eps \} \geq 1 - \eps.
\end{equation}

The compact containment condition~\eqref{FL:eq:compact} follows easily by the upper bound
\[
\| \overline{\vc{Q}}^{\,r}(T) \|_1 \leq \| \overline{\vc{Q}}^{\,r}(0) \|_1 + A(rT)/r \Rightarrow \|\vc{z}(0)\|_1 + \lambda T \quad \text{as $r \to \infty$}.
\]
Take an $\wtil{M} < \infty$ that is a continuity point for the distribution of $\|\vc{z}(0)\|_1$ such that $\pr \{ \|\vc{z}(0)\|_1 \leq \wtil{M} \} \geq 1 - \eps$ and put $M = \wtil{M} + \lambda T + 1$.

To establish the oscillation control condition~\eqref{FL:eq:oscillation}, it is enough to have oscillations of the scaled departure processes $D_i^{r,\textup{s}}(r \cdot)/r$ and $D_i^{r,\textup{a}}(r \cdot)/r$ bounded.

Define the modulus of continuity for functions $x \colon \nneg \to \real$,
\[
\omega(x,T,\delta) := \sup \{ |x(s) - x(t)| \colon s,t \in [0,T],  |s-t|<\dlt \}.
\]

First we estimate oscillations of $D_i^{r,\textup{s}}(r \cdot)/r$. We have, for all $t \geq s \geq 0$,
\begin{align*}
&\left| \dfrac{D_i^{r,\textup{s}}(r s)}{r} - \dfrac{D_i^{r,\textup{s}}(r t\,)}{r} \right| \leq \, |G^{r,\textup{s}}_i(s)| + |G^{r,\textup{s}}_i(t)| + \mu_i \int_s^t \dfrac{\overline{Q}_i^{\, r}(u)}{\| \overline{\vc{Q}}^{\, r}(u) \|_1} \,du,
\end{align*}
where, for all $t \in \nneg$,
\begin{equation} \label{FL:eq:10}
\begin{split}
G^{r,\textup{s}}_i(t) :=  \dfrac{1}{r} \Pi_i^\textup{s}\left( r\mu_i \int_0^t \dfrac{\overline{Q}_i^{\, r}(u)}{\| \overline{\vc{Q}}^{\, r}(u) \|_1} \,du\right) 
- \mu_i \int_0^t \dfrac{\overline{Q}_i^{\, r}(u)}{\| \overline{\vc{Q}}^{\, r}(u) \|_1} \,du.
\end{split}
\end{equation}
Then
\begin{equation} \label{FL:eq:6}
\omega\left(\frac{D_i^{r,\textup{s}}(r \cdot)}{r},T,\delta\right)  \leq \ 2 \sup_{t \in [0,\mu_i T]} \left|\dfrac{\Pi_i^\textup{s}(rt)}{r} - t\right| + \delta.
\end{equation}

Now we switch to $D_i^{r,\textup{a}}(r \cdot)/r$. Consider a family of $M/M/\infty$ queues with a common arrival process $A(\cdot)$, queue~$r$ starting with $\|\vc{Q}^r(0)\|_1$ customers, and service times in queue~$r$ being patience times of the corresponding customers in the $r$-th PS-queue with multistage service. Denote the departure process of the $r$-th $M/M/\infty$-queue by $\wtil{D}^r(\cdot)$. We have, for all $i$ and $s,t \in \nneg$,
\begin{equation*}
\left| \dfrac{D_i^{r,\textup{a}}(r s)}{r} - \dfrac{D_i^{r,\textup{a}}(r t)}{r} \right| \leq \left| \dfrac{\wtil{D}^r(rs)}{r}  - \dfrac{\wtil{D}^r(r t)}{r} \right|,
\end{equation*}
and hence,
\begin{equation}  \label{FL:eq:7}
\omega(D^{r,\text{a}}(r\cdot)/r,T,\delta) \leq \omega(\wtil{D}^r(r\cdot)/r,T,\delta).
\end{equation}

By e.g.\ Robert \cite{Robert}, the scaled processes $\wtil{D}^r(r\cdot)/r$ converge weakly in the Skorokhod space $\mathbf{D}(\nneg, \nneg)$ to a continuous limit, which we denote by $\wtil{D}(\cdot)$. (Although technically the fluid scalings considered in Robert \cite{Robert} and here are different: arrival rates and space versus time and space, they result in the same distributions of the scaled processes, and hence the same limit.)

Since the modulus of continuity $\omega(\cdot,T,\dlt)$ as a function on $\mathbf{D}(\nneg,\real)$ is continuous at any continuous $x(\cdot)$, we have, by the continuous mapping theorem,
\begin{equation} \label{FL:eq:8}
\omega(\wtil{D}^r(r\cdot)/r,T,\delta) \Rightarrow \omega(\wtil{D}(\cdot),T,\delta) \quad \text{as $r \to \infty$}.
\end{equation}
Since continuity implies uniform continuity on compact sets, we also conclude that
\begin{equation} \label{FL:eq:9}
\omega(\wtil{D}(\cdot),T,\delta) \Rightarrow 0 \quad \text{as $\delta \to \infty$}
\end{equation}

Finally, as we put together the FLLN for $A(\cdot)$, \eqref{FL:eq:6} and the FLLN for $\Pi_i^\textup{s}(\cdot)$, and also~\eqref{FL:eq:7}--\eqref{FL:eq:9}, it follows that one can pick a~$\delta$ such that~\eqref{FL:eq:oscillation} holds.

\paragraph{Fluid limits as FMS's} Now that we know that fluid limits exist, it is left to check that they are FMS's. Consider a fluid limit $\wtil{\vc{Q}}(\cdot)$ along a subsequence $\{\overline{\vc{Q}}^{\, q}(\cdot)\}_{q \to \infty}$. The $\mathbf{C}$-tightness part of the proof implies that $\wtil{\vc{Q}}(\cdot)$ is a.s.\ continuous. As we discussed before, the total population process of a PS-queue with multistage service behaves as an ordinary, single-stage-service PS-queue, whose fluid limits are studied by Gromoll et al. \cite{GRZ08}. In particular, it follows from Assumption~\ref{FL:ass:load} and~\cite[Lemma~6.1]{GRZ08} that a.s., for all $\delta>0$, $\inf_{t \geq \delta}\|\wtil{\vc{Q}}(t)\|_1 > 0$. We will now show that $\wtil{\vc{Q}}(\cdot)$ a.s.\ satisfies the fluid model equations~\eqref{FL:eq:fluid_model_dif}, and this will finish the proof.

Consider the mappings $\varphi_i \colon \mathbf{D}(\nneg,\nneg^I) \to \mathbf{D}(\nneg,\real)$, $i = 1, \ldots, I$, given by
\begin{align*}
\varphi_1(\vc{x})(t) =:&\ x_1(t) - x_1(0) - \lambda t + \mu_1 \int_0^t \dfrac{x_1(u)}{\|\vc{x}(u)\|_1}\,du + \nu \int_0^t x_1(u) \, du, \\
\varphi_i(\vc{x})(t) =:&\ x_i(t) - x_i(0) - \mu_{i-1} \int_0^t \dfrac{x_{i-1}(u)}{\|\vc{x}(u)\|_1}\,du + \mu_i \int_0^t \dfrac{x_i(u)}{\|\vc{x}(u)\|_1}\,du + \nu \int_0^t x_i(u) \, du, \quad i \geq 2.
\end{align*}

These mappings are continuous at any $\vc{x}(\cdot)$ that is continuous and non-zero outside $t = 0$. Then, by the continuous mapping theorem, for all~$i$,
\begin{equation} \label{FL:eq:14}
\varphi_i(\overline{\vc{Q}}^q) \Rightarrow \varphi_i(\wtil{\vc{Q}}) \quad \text{as $q \to \infty$}.
\end{equation}

On the other hand, it follows from the stochastic dynamics~\eqref{FL:eq:dyn_1_r}--\eqref{FL:eq:dyn_2_r} that, for all $q$ and $t \in \nneg$,
\begin{equation} \label{FL:eq:11}
\begin{split}
\varphi_1(\overline{\vc{Q}}^q)(t) &= (A(qt)/q - \lambda t) - G^{\,q,\textup{s}}_1(t) - G^{\,q,\textup{a}}_1(t), \\
\varphi_i(\overline{\vc{Q}}^q)(t) &= G^{\,q,\textup{s}}_{i-1}(t) - G^{\,q,\textup{s}}_i(t) - G^{\,q,\textup{a}}_i(t), \quad i \geq 2,
\end{split}
\end{equation}
where, for all $i$ and $t \in \nneg$,
\[
G^{\,q,\textup{a}}_i(t) := \dfrac{1}{q} \Pi_i^\textup{a}\left( q \nu \int_0^t \overline{Q}_i^{\, q}(u) \,du\right) - \nu \int_0^t \overline{Q}_i^{\, q}(u)\,du,
\]
and the processes $G^{\,q,\textup{s}}_i(\cdot)$ were defined earlier by~\eqref{FL:eq:10}.

Next we use the following result (see e.g.\ Billingsley \cite{Billingsley}).

\begin{proposition}[{\bf Random time change theorem}] \label{FL:prop:random_time_change}
Consider stochastic processes $X^q(\cdot) \in \mathbf{D}(\nneg,S)$, where $S$ is a complete and separable metric space, and non-decreasing stochastic processes $\Phi^q(\cdot) \in \mathbf{D}(\nneg,\nneg)$. Assume that the joint convergence $(X^q, \Phi^q)(\cdot) \Rightarrow (X, \Phi)(\cdot)$ holds as $q \to \infty$, and that the limits $X(\cdot)$ and $\Phi(\cdot)$ are a.s.\ continuous. Then $X^q(\Phi^q(\cdot)) \Rightarrow X(\Phi(\cdot))$ in $\mathbf{D}(\nneg,S)$ as $q \to \infty$.
\end{proposition}

Put $X^q(t) = \Pi_i^\textup{s}(qt)/q - t$ and $\Phi^q(t) = \nu \int_0^t \overline{Q}_i^{\, q}(u)\,du$ for all~$t \in \nneg$. The marginal weak limits of these processes are $X(\cdot) \equiv 0$ and $\Phi(\cdot) = \nu \int_0^\cdot \wtil{Q}_i(u)\,du$, respectively. Since one of the marginal limits is deterministic, we actually have the joint weak convergence, and then Proposition~\ref{FL:prop:random_time_change} implies that, as $q \to \infty$,
\begin{equation} \label{FL:eq:12}
G_i^{\,q, \textup{a}}(\cdot) \Rightarrow 0 \quad \text{in $\mathbf{D}(\nneg,\real)$}.
\end{equation}
Similarly, 
\begin{equation} \label{FL:eq:13}
G_i^{\,q, \textup{s}}(\cdot) \Rightarrow 0 \quad \text{in $\mathbf{D}(\nneg,\real)$}.
\end{equation}

As we put \eqref{FL:eq:11}--\eqref{FL:eq:13} together with \eqref{FL:eq:14}, it follows that
\[
\text{a.s., for all $i$}, \quad \varphi_i(\wtil{\vc{Q}}) \equiv 0,
\]
which, after differentiation, gives~\eqref{FL:eq:fluid_model_dif}.


\end{document}